\documentclass[reqno,11pt,a4paper]{amsart}
\usepackage{amssymb}
\usepackage{amsthm}
\usepackage{microtype}
\usepackage{enumitem}
\usepackage{lmodern}
\usepackage{mathtools}
\usepackage[usenames,dvipsnames]{xcolor}
\definecolor{ovgu_red}{rgb}{0.82,0.25,0.35}
\usepackage{hyperref}
\hypersetup{
 colorlinks,
 linkcolor={red!50!black},
 citecolor={blue!50!black},
 urlcolor={blue!80!black}
}
\usepackage[a4paper, inner=3cm,outer=3cm,top=3.3cm,bottom=3cm]{geometry}
\usepackage{subcaption}
\usepackage{todonotes}  
\usepackage{graphicx}
\usepackage{tikz,tikz-3dplot}
\usetikzlibrary{cd}
\usetikzlibrary{shapes.misc}
\usetikzlibrary{arrows,decorations.markings}
\tikzset{%
	every picture/.style = {scale=0.5},
	newton polytope/.style = {%
		line width = 1pt,
		every circle/.style = {radius=0.1},
		lattice point/.style = {color=gray!50},
	},
}
\usepackage{xifthen}

\AtBeginDocument{%
\def\MR#1{}
}
\theoremstyle{plain}
\newtheorem{thm}{Theorem}[section]
\newtheorem{prop}[thm]{Proposition}

\newtheorem{lem}[thm]{Lemma}

\theoremstyle{definition}
\newtheorem{defn}[thm]{Definition}
\newtheorem{ex}[thm]{Example}
\newtheorem{rem}[thm]{Remark}

\newcommand{\supp}{{\mathrm{supp}}}

\newcommand{\sqp}[1][]{%
\ifthenelse{\isempty{#1}}{\mathcal{P}(\square_2)}{\mathcal{P}^{#1}(\square_2)}%
}

\DeclareMathOperator{\conv}{conv}

\DeclareMathOperator{\MV}{MV}

\DeclareMathOperator{\NP}{N}

\DeclareMathOperator{\Ht}{ht}

\newcommand{\R}{{\mathbb{R}}}
\newcommand{\Z}{{\mathbb{Z}}}
\newcommand{\C}{{\mathbb{C}}}

\newcommand{\Sm}{{\mathcal{S}}}

\newcommand{\Rm}{{\mathcal{R}}}
\newcommand{\Qm}{{\mathcal{Q}}}

\newcommand{\rleft}{\mathopen{}\mathclose\bgroup\left}
\newcommand{\rright}{\aftergroup\egroup\right}

\newcommand{\set}[1]{\rleft\{ {#1} \rright\}}

\DeclareMathOperator{\Vol}{Vol}

\def\coloneqq{\mathrel{\mathop:}=}
\title{Liftings of polynomial systems decreasing the mixed volume}
\author[C.\,Borger]{Christopher Borger}
\address[C.\,Borger]{Fakult\"at f\"ur Mathematik\\Institut f\"ur Algebra und Geometrie\\Otto-von-Guericke-Universit\"at Magdeburg\\Universit\"atsplatz 2\\ 39106 Magdeburg\\Germany}
\curraddr{}
\email{christopher.borger@ovgu.de}
\thanks{}
\author[T.\,Kahle]{Thomas Kahle}
\address[T.\,Kahle]{Fakult\"at f\"ur Mathematik\\Institut f\"ur Algebra und Geometrie\\Otto-von-Guericke-Universit\"at Magdeburg\\Universit\"atsplatz 2\\ 39106 Magdeburg\\Germany}
\curraddr{}
\email{thomas.kahle@ovgu.de}
\thanks{}
\author[A.\,Kretschmer]{Andreas Kretschmer}
\address[A.\,Kretschmer]{Fakult\"at f\"ur Mathematik\\Institut f\"ur Algebra und Geometrie\\Otto-von-Guericke-Universit\"at Magdeburg\\Universit\"atsplatz 2\\ 39106 Magdeburg\\Germany}
\curraddr{}
\email{andreas.kretschmer@ovgu.de}
\thanks{}
\author[S.\,Sager]{Sebastian Sager}
\address[S.\,Sager]{Fakult\"at f\"ur Mathematik\\Institut f\"ur Mathematische Optimierung\\Otto-von-Guericke-Universit\"at Magdeburg\\Universit\"atsplatz 2\\ 39106 Magdeburg\\Germany}
\curraddr{}
\email{sager@ovgu.de}
\thanks{}
\author[J.\,Schulze]{Jonas Schulze}
\address[J.\,Schulze]{Fakult\"at f\"ur Mathematik\\Institut f\"ur Mathematische Optimierung\\Otto-von-Guericke-Universit\"at Magdeburg\\Universit\"atsplatz 2\\ 39106 Magdeburg\\Germany}
\curraddr{}
\email{jonas.schulze@ovgu.de}
\thanks{This work is funded by the Deutsche Forchungsgemeinschaft (314838170,
  GRK 2297, MathCoRe).}

\makeatletter
\@namedef{subjclassname@2020}{\textup{2020} Mathematics Subject Classification}
\makeatother
\subjclass[2020]{Primary: 52A39, 65H20; Secondary: 65H04, 14M25, 13P15, 14Q65, 52B20}

\keywords{lattice polytopes, sparse polyhedral homotopy, mixed volume}

\begin{document}

\begin{abstract}
  The BKK theorem states that the mixed volume of the Newton polytopes
  of a system of polynomial equations upper bounds the number of
  isolated torus solutions of the system.  Homotopy continuation
  solvers make use of this fact to pick efficient start systems.  For
  systems where the mixed volume bound is not attained, such methods
  are still tracking more paths than necessary.  We propose a strategy
  of improvement by lifting a system to an equivalent system with a
  strictly lower mixed volume at the expense of more variables.  We
  illustrate this idea providing lifting constructions for arbitrary
  bivariate systems and certain dense-enough systems.
\end{abstract}

\maketitle{}

\section{Introduction}

The Bernstein–Khovanskii–Kouchnirenko (BKK) theorem states that the
number of isolated solutions of a system of $d$ Laurent polynomials in
$d$ variables in the algebraic torus $(\C^*)^d$ is bounded above by
the mixed volume of the Newton polytopes of the
polynomials~\cite{B75}, \cite[Section~7.5]{CLO05}.  This bound only
depends on the monomials occurring in the different polynomials and
may or may not be attained by a concrete choice of coefficients.
Whether the mixed volume bound is attained depends on the solvability
of lower-dimensional systems determined by the convex geometry of the
Newton polytopes~\cite{B75}.

Homotopy continuation is an important technique to solve polynomial
systems by tracking the solutions of a known start system to the
solutions of a system of interest along the paths of a homotopy.  The
computational complexity of this method depends on the number of paths
to be tracked and the cost of tracking each path.  It can be difficult
to judge how the two costs are related.  Ideally one wants to track
only as many paths as the system of interest has solutions.  In
practical applications that number is almost always much lower than
the number of solutions of the start system.  In total, a homotopy
continuation method needs a good upper bound on the number of
solutions and a suitable start system.  In \cite{huberSturmfels},
Huber and Sturmfels show how to construct efficient start systems
using the geometry of the Newton polytopes, introducing mixed
subdivisions for sparse homotopy continuation.  Solvers based on this
approach are the current state of the art and are implemented in
\texttt{PHCpack}~\cite{phcpack} and as
\texttt{HomotopyContinuation.jl} in \texttt{Julia}~\cite{hcJulia}.
The number of paths to be tracked in a sparse homotopy is the mixed
volume of the Newton polytopes of the target system.  In practice,
however, the mixed volume bound is often not attained and therefore
sparse homotopy continuation still tracks superfluous paths that do
not lead to actual torus solutions of the system.

In this paper we propose to study \emph{liftings} of polynomial
systems to equivalent systems with more equations and more variables
but fewer paths to be tracked.  After laying out the idea, we present
methods for liftings that work for certain dense-enough systems of
polynomials (Theorem~\ref{thm:main}), system with linear dependency in
a facial subsystem (Theorem~\ref{thm:linearly_dependent}), and for
bivariate systems (Theorem~\ref{thm:bivariate_lifting}).

It is not easy to judge the implications of lifting to the complexity
of solving a polynomial system.  The celebrated solution of Smale's
17th problem~(see \cite{lairez2017deterministic} and its references)
shows that a path in a homotopy can, on average with respect to a
suitable distribution of input systems, be tracked in polynomial time.
Finding one solution of a polynomial system is tracking such a path.
It is an interesting challenge to investigate the effects of lifting
to the computation of individual solutions.  First steps with respect
to just simple Newton iterations have been undertaken in
\cite{albersmeyer2010lifted}.  In the present paper we take the
standpoint that \emph{solving} means to determine all solutions to a
system.  For a complete analysis of the complexity of solving, one
would need to study the relation of the costs of superfluous paths
versus a more expensive tracking of fewer paths.

We begin by illustrating our lifting approach in a simple example.
\begin{ex}
  \label{ex:motivating_example}
  Consider the system
  \begin{align}
    \begin{split}
      \label{eq:ex_sys}
      0 &= (1 - x_1^2) x_2 + 2, \\
      0 &= (1 - x_1)^2 x_2 + 3.
    \end{split}
  \end{align}
  We aim to solve this system in~$(\C^\ast)^2$. The Newton polytopes
  of both polynomials agree with $P = \conv(0,e_2,2e_1+e_2)$
  (see Figure~\ref{fig:polytope1}).
\begin{figure}
	\begin{tikzpicture}[newton polytope]
		\foreach \x in {0,1,2,3}{
			\foreach \y in {0,1,2}{
				\fill[lattice point] (\x,\y) circle;
			}
		}
		\draw (0,1) -- (0,0) -- (2,1);
		\draw[ovgu_red] (2,1) -- (0,1);
		\foreach \pos in {(0,0), (0,1), (1,1), (2,1)}{
			\fill \pos circle;
		}
	\end{tikzpicture}
	\caption{The polytope $P$ of
          Example~\ref{ex:motivating_example} with the critical facial
          subsystem in red.}
	\label{fig:polytope1}
\end{figure}
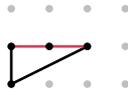
The mixed volume equals $2$.  However, in $(\C^{\ast})^{2}$ the system
has a unique solution.  According to \cite{B75}, this happens exactly if a
\emph{facial subsystem} has a root in $(\C^*)^2$ (see
Theorem~\ref{t:bernstein_crit}).  In our case this facial subsystem
corresponds to the red face in Figure~\ref{fig:polytope1}:
\begin{align}
\begin{split}
\label{eq:ex_fac_sub}
0 &= (1-x_1^2)x_2 \\
0 &= (1-x_1)^2x_2,
\end{split}
\end{align}
and has infinitely many solutions $\set{(1,x_{2}) : x_{2}\in\C^{*}}$.
Now consider the system
\begin{align}
\begin{split}
\label{eq:ex_sys_lift}
0 &= y(1+x_1)x_2 +2, \\
0 &= y(1-x_1)x_2 + 3, \\
0 &= y - (1 - x_1).
\end{split}
\end{align}
Solving the last equation for $y$ and plugging in, recovers the
original system.  Therefore the lifted system is equivalent to the
original.  In contrast, the lifted system has mixed volume
equal to~$1$, the true number of torus solutions.  What happens is
that the original system has a zero at infinity, namely $(1, \infty)$,
which can be seen by solving for~$x_2$.  This pseudo-solution is taken
care of in the lifting.  Homotopy continuation for the original system
involves tracking a path tending to infinity, while for the lifted
system this path is ignored.
\end{ex}

In order to illustrate our main lifting strategy and the necessary
tools, we discuss further details of
Example~\ref{ex:motivating_example}.
After division by $x_2$ one observes that the facial subsystem
\eqref{eq:ex_fac_sub} has a solution in $(\C^*)^2$ if and only if the
univariate polynomials $(1 - x_1^2)$ and $(1 - x_1)^2$ have a common
non-zero root, in our case~$x_{1}=1$.  Our approach is to reflect this
information in a Newton polytope, so that it contributes to the mixed
volume computation.  We achieve this by means of a polynomial division
in order to substitute the linear factor $(1-x_1)$ by a new
variable~$y$.  In the case of a linear factor, the effect of the
division on Newton polytopes, the \emph{polytope division}, can be
worked out and is a main ingredient in our method (see
Definition~\ref{d:PolyQuotient}).  In this way we arrive at an
equivalent lifted system~\eqref{eq:ex_sys_lift}.
Definition~\ref{d:lift} contains the details of lifting.
In the lifted system, the non-genericity that $y = (1-x)$ appears as a
common factor in both polynomials of \eqref{eq:ex_sys} is represented
also polyhedrally:
In this sense we view \eqref{eq:ex_sys} in a higher-dimensional space
of sparse polynomial systems, in which
having strictly
fewer solutions than the mixed volume is not encoded in the
coefficients, but in the geometry of the Newton polytopes.  Hence, the
mixed volume of the Newton polytopes of the system
\eqref{eq:ex_sys_lift} is strictly less than that of the Newton
polytopes of~\eqref{eq:ex_sys}.  The embedding of information about
the number of solutions in the convex geometry of Newton polytopes can
only work if the construction of the lifting depends on the
coefficients of the system~\eqref{eq:ex_sys}.
Theorem~\ref{thm:bivariate_lifting} shows that in the bivariate case
there always exist lifts generalizing
Example~\ref{ex:motivating_example} as long as the mixed volume bound
is not attained.

Theorem~\ref{thm:main} generalizes the above approach to systems of
$d$ polynomials in $d$ variables, but the exact lifting strategy does
not carry over.  In higher dimension the lifting becomes more involved
and more conditions are required.  The replacement of common linear
terms by new variables is implemented by polynomial division and
additional assumptions are necessary for the polyhedral geometry to
work.  A strong but natural additional assumption is that the Newton
polytopes of the subsystem for which Bernstein's criterion fails are
\emph{saturated} in the sense of Definition~\ref{d:saturated}.  The
content of this definition is best seen in
Lemma~\ref{lem:sparse_division}, which shows that a sparse system
behaves polyhedrally well under division by linear univariate terms if
and only if it is saturated.  Example~\ref{ex:poly_div} illustrates
that this well-behavedness is necessary for the lifting to work.  In
order to ensure that the mixed volume of the lifting is strictly
lower, we additionally need to enforce that the Newton polytopes of
the original system are not degenerate in a certain sense (see
Lemma~\ref{lem:technical_lemma_from_hell}).



Our work describes individual lifting steps, reducing the mixed volume
by one in Theorems~\ref{thm:main} and \ref{thm:linearly_dependent} or
the degree of a $\gcd$ in Theorem~\ref{thm:bivariate_lifting}.  Of
course one would like to be able to make several improvements.  Most
obviously it would be nice to reduce the mixed volume more by
constructing refined lifts.  If this is not possible, one would like
to lift iteratively, ideally until the mixed volume bound is sharp.
At the moment it is an open problem to determine how iterative
liftings work.
%



\subsection*{Acknowledgment} We thank Ivan Soprunov for helpful
clarifications and motivating feedback.  Paul Breiding explained
practical aspects of homotopy continuation to the authors.

\section{Preliminaries}

We abbreviate $[n] \coloneqq \{1, \ldots, n\}$. A \emph{lattice
  polytope} $P \subseteq \R^d$ is a convex polytope all of whose
vertices lie in $\Z^d$. We denote by $\Vol(P)$ the normalized volume
of $P$, which equals the standard euclidean volume multiplied by~$d!$.
The \emph{(normalized) mixed volume} of $d$-tuples of lattice
polytopes living in $\R^d$ is the unique functional satisfying
\begin{align*}
\Vol(\lambda_1 P_1 + \dots + \lambda_k P_k)
= \sum_{i_1=1}^k \dots \sum_{i_d=1}^k \lambda_{i_1} \dots \lambda_{i_d}
\MV(P_{i_1},\dots,P_{i_d}),
\end{align*}
for all choices of lattice polytopes $P_1,\dots,P_k \subset \R^d$,
non-negative scalars $\lambda_1,\dots,\lambda_k \geq 0$, and
$k \in \Z_{\geq 1}$. We refer to of \cite[Theorem and
Definition~5.1.7]{Schn14} for a detailed treatment of existence and
uniqueness of such a functional.

\subsection{The BKK theorem}
\label{sec:bkk}

Let $f \in \C[x_1^{\pm 1},\dots,x_d^{\pm 1}]$.  The \emph{support} of $f$ is
\begin{align*}
\supp f \coloneqq \{a \in \Z^d \colon \text{ the monomial
 } x^a \text{ has non-zero coefficient in } f\},
\end{align*}
and its convex hull $\NP(f) = \conv(\supp(f))$ is the \emph{Newton
  polytope} of~$f$. A \emph{monomial change of variables} is a map
transforming a system of Laurent polynomials by sending
$x_i \mapsto x^{u_i}$ for all $i \in [d]$ in each of the polynomials,
where the $u_i$ form the columns of a unimodular integer matrix $U$,
followed by multiplication of each polynomial by some monomial. Any
monomial change of variables induces a bijection between the torus
solutions of the original system and its transformation.  The
following is essentially~\cite[Theorem~A]{B75}.

\begin{thm}\label{t:bkk}
Let $f_1,\dots,f_d \in \C[x_1^{\pm 1},\dots,x_d^{\pm 1}]$ be Laurent
polynomials. Then the number of solutions of the system
$f_1 = \dots = f_d = 0$ in $(\C^*)^d$ is bounded from above by
$\MV(\NP(f_1),\dots,\NP(f_d))$.
\end{thm}

In order to formulate the conditions under which the above bound is
attained, we introduce the notion of a facial system.  For any
polytope $P \subset \R^d$ and any non-zero vector $u \in \R^d$, we
denote by $P^u$ the face of $P$ maximizing the functional
$\langle \cdot , u \rangle$.  For a polynomial
$f \in \C[x_1,\dots,x_d]$, we denote by $f^u$ the polynomial
consisting only of those terms of $f$ that are supported on the
face~$\NP(f)^u$.  Now this is essentially~\cite[Theorem~B]{B75}.

\begin{thm} \label{t:bernstein_crit}
Let $f_1,\dots,f_d \in \C[x_1^{\pm 1},\dots,x_d^{\pm 1}]$ be Laurent
polynomials. Then the number of solutions of the system
$f_1 = \dots = f_d = 0$ in $(\C^*)^d$ equals
$\MV(\NP(f_1),\dots,\NP(f_d))$
(counting multiplicities) if and only if for any $u \in \R^d \setminus \{0\}$ the
facial system $f_1^u=\dots=f_d^u=0$ has no solution in $(\C^*)^d$.
\end{thm}

Up to an appropriate monomial change of variables, the facial system
corresponding to any $u \in \R^d \setminus \{0\}$ is a system of $d$
equations in $d-1$ variables (as the corresponding Newton polytopes
live in parallel $(n-1)$-dimensional hyperplanes orthogonal to $u$).
In particular, generically, none of these systems has a solution and
the BKK bound is attained.

\subsection{Liftings of sparse polynomial systems}
For point configurations $A_1,\dots,A_d \subseteq \Z^{d}$,
we denote by $\C[A_1,\dots,A_d]$ the vector subspace of
$\left(\C[x_1^{\pm 1}, \ldots, x_d^{\pm 1}]\right)^{d}$ consisting of all
$(f_1, \ldots, f_d)$ with $\supp(f_i) \subseteq A_i$ for all $i \in [d]$.
For a tuple of lattice polytopes $(P_1, \ldots, P_d)$ we set
$\C[P_1, \ldots, P_d] \coloneqq \C[P_1 \cap \Z^d, \ldots, P_d \cap \Z^d]$.

\begin{defn}[Lifting of a system]
\label{d:lift}
For a tuple
$F=(f_1,\dots,f_d) \in (\C[x_1^{\pm 1},\dots,x_d^{\pm 1}])^d$,
\begin{align*}
\tilde{F}=(\tilde{f}_1,\dots,\tilde{f}_d,
h_1,\dots,h_k) \in (\C[x_1^{\pm 1},\dots,x_d^{\pm 1},y_1,
\dots,y_k])^{d+k}
\end{align*}
is a \emph{lifting} of the system $f_1 = \dots = f_d = 0$, if a point
$(\alpha_1,\dots,\alpha_d) \in (\C^*)^d$ is a solution of $F$ if and only if there exists a solution
$(\alpha_1,\dots,\alpha_d,\beta_1,\dots,\beta_k) \in (\C^*)^{d} \times
\C^k$ of $\tilde{F}$.
\end{defn}

\begin{rem}
  For our purposes it is important to allow zeroes in the last $k$
  coordinates of the roots of the lifted system.  See
  Remark~\ref{rem:caveat} for possible caveats in lifting strategies.
\end{rem}

For any tuple $(f_1,\dots,f_d) \in (\C[x_1^{\pm 1},\dots,x_d^{\pm 1}])^d$
we denote by $\Sm(f_1,\dots,f_d)$ its number of isolated solutions
(counting multiplicities) in the complex torus $(\C^*)^d$.

\begin{defn}[Solution-preserving map]
\label{d:solution-preserving}
Let $A_1,\dots,A_{d_1} \subset \Z^{d_1}$ and
$B_1,\dots,B_{d_2} \subset \Z^{d_2}$.  For a subset
$U \subset \C[A_1,\dots,A_{d_1}]$, we call a map
$\phi \colon U \to \C[B_1,\dots,B_{d_2}]$ \emph{solution-preserving},
if for any tuple $(f_1,\dots,f_{d_1})\in U$, one has
$\Sm((f_1,\dots,f_{d_1}))=\Sm(\phi(f_1,\dots,f_{d_1}))$.
\end{defn}

The following is a direct consequence of Theorem~\ref{t:bkk}.
\begin{prop} \label{prop:trivial_lift} Let
  $A_1,\dots,A_{d_1} \subset \Z^{d_1}$ and
  $B_1,\dots,B_{d_2} \subset \Z^{d_2}$. Let
  $U \subset \C[A_1,\dots,A_{d_1}]$ and
  $V \subset \C[B_1,\dots,B_{d_2}]$ be Zariski-open dense sets. If
  there exists a solution-preserving map
  $\phi_1 \colon U \to \C[B_1,\dots,B_{d_2}]$, then one has
  $\MV(A_1,\dots,A_{d_1}) \leq \MV(B_1,\dots,B_{d_2})$.  In
  particular, if there additionally exists a solution-preserving map
  $\phi_2 \colon V \to \C[A_1,\dots,A_{d_1}]$, one has
  $\MV(A_1,\dots,A_{d_1})=\MV(B_1,\dots,B_{d_2})$.
\end{prop}

\subsection{Monotonicity of the Mixed Volume}
\begin{defn}
\label{d:essentiality}
A tuple of lattice polytopes $P_1,\dots,P_d \subset \R^d$ is
\emph{essential} if one of the following equivalent conditions holds:
\begin{enumerate}
\item \label{it:mv>0} $\MV(P_1,\dots,P_d) > 0$,
\item there exists a choice of segments $s_i \subseteq P_i$
for $i \in [d]$ such that $s_1,\dots,s_d$ are linearly
independent,
\item \label{it:trans} there is no subset $\emptyset \neq I \subseteq [d]$ for
which the polytopes $\{P_i \colon i \in I \}$ can be translated
to a common $(|I|-1)$-dimensional linear subspace of $\R^d$,
\item \label{it:mink} there is no subset $\emptyset \neq I \subseteq [d]$
satisfying
\begin{align*}
\dim \left(\sum_{i \in I} P_i \right) \leq |I|-1.
\end{align*}
\end{enumerate}
\end{defn}

The equivalence of \eqref{it:trans} and \eqref{it:mink} is
straightforward to verify and \cite[Theorem~5.1.8]{Schn14} contains a
proof of the remaining equivalences.

It is well-known that the mixed volume is inclusion-monotonous in each
argument.  It is, however, not strictly monotonous as one can replace
several polytopes in a tuple by strictly smaller ones without
decreasing the mixed volume.  The following theorem of Bihan and
Soprunov is a characterization of when the mixed volume of a subtuple
is strictly smaller than that of the original tuple.  It is a crucial
tool in the proofs of both our main results Theorem~\ref{thm:main} and
Theorem~\ref{thm:linearly_dependent}.

\begin{thm}{\cite[Theorem~3.3]{BS19}}
\label{thm:bihan-soprunov}
Let $P_1,\dots,P_d \subseteq \R^d$ and
$P_1',\dots,P_d' \subseteq \R^d$ tuples of polytopes such that
$P_i' \subseteq P_i$ for every $i \in [d]$. Given $u \in \R^d$, denote
$T_u = \{i \in [d] \colon P_i' \text{ touches } P_i^u \}$. Then
$\MV(P_1,\dots,P_d) > \MV(P_1',\dots,P_d')$ if and only if there
exists a $u \in \R^{d}$ such that the tuple
$\{P_i^u \colon i \in T_u \} \cup \{P_i \colon i \in [d] \setminus T_u
\}$ is essential.
\end{thm}

\section{Liftings via polynomial division}

\begin{defn}\label{d:PolyQuotient}
  Let $P \subset \R_{\geq 0}^d$ be a lattice polytope.  The \emph{quotient}
  $\Qm_i(P)$ and \emph{remainder} $\Rm_i(P)$ of $P$ are
\begin{align*}
  \Qm_i(P) &= \conv(x - e_i, x - \langle x,e_i \rangle e_i \colon x \in P \cap \Z^d, \langle x , e_i \rangle > 0), \\
  \Rm_i(P) &= \conv(x - \langle x,e_i \rangle e_i \colon x \in P \cap \Z^d).
\end{align*}
\end{defn}

\begin{lem}
\label{lem:sparse_division}
Let $f,q,r \in \C[x_1,\dots,x_d]$ such that $q$ and $r$ are quotient
and remainder, respectively, of the polynomial division of $f$ by
$(x_i-\alpha)$ for some $\alpha \in \C^\ast$. Then
$N(q) \subseteq \Qm_i(N(f))$ and $N(r) \subseteq \Rm_i(N(f))$.
\end{lem}

\begin{proof}
  For any monomial $x_1^{a_1} \cdots x_i^{a_i} \cdots x_d^{a_d}$ with
  $a_i > 0$, polynomial division by $(x_i - \alpha)$ gives
  \begin{align*}
    x_1^{a_1} \cdots x_i^{a_i} \cdots x_d^{a_d} =
    (x_1^{a_1} \cdots x_i^{a_i-1} \cdots x_d^{a_d})(x_i - \alpha)
    + \alpha (x_1^{a_1} \cdots x_i^{a_i-1} \cdots x_d^{a_d}).
  \end{align*}
  Therefore, for every monomial
  $x_1^{a_1} \cdots x_i^{a_i} \cdots x_d^{a_d}$ occurring in $f$, the
  division algorithm can only produce monomials of the form
  $x_1^{a_1} \cdots x_i^{b_i} \cdots x_d^{a_d}$ for $0 \leq b_i < a_i$ in $q$, showing $N(q) \subseteq \Qm_i(N(f))$.\\
  The only monomials which cannot be divided by $(x_i - \alpha)$ are
  those in which $x_i$ does not occur.  So the only monomials
  which can occur in the remainder $r$ are of the form
  $x_1^{a_1} \cdots \widehat{x_i^{a_i}} \cdots x_d^{a_d}$ where
  $x_1^{a_1} \cdots x_i^{a_i} \cdots x_d^{a_d}$ occurs in $f$ for some
  $a_i \geq 0$. This shows $N(r) \subseteq \Rm_i(N(f))$.
\end{proof}

For a generic choice of coefficients of $f$ in $\C[\supp(f)]$, it can
be shown that equality holds in both cases of
Lemma~\ref{lem:sparse_division}.  The proof consists of analyzing in
each step of polynomial division which conditions on the coefficients
accumulate.

\begin{ex}
\label{ex:poly_div}
Consider a polynomial of the form
\begin{align*}
f(x_1,x_2) = a + b x_1^2 + c x_1^2 x_2^2.
\end{align*}
Division of $f$ by a term $(x_1 - \alpha)$ results in a decomposition
\begin{align*}
f(x_1,x_2) &= (c x_1x_2^2+ \alpha c x_2^2
+ b x_1 + \alpha b)(x_1 - \alpha) +
\alpha^2 c x_2^2 + \alpha^2 b + a \\
&= Q(x_1,x_2)(x_1 - \alpha) + R(x_2),
\end{align*}
where one has $\NP(Q) = \Qm_1(\NP(f))$ and
$\NP(R)= \Rm_1(\NP(f))$ (if $a,b,c$ are sufficiently general).  For different polynomials
$Q^\prime(x_1,x_2) \in \C[\Qm_1(\NP(f))]$ and $R^\prime(x_1) \in \C[\Rm_1(\NP(f))]$ with the same Newton polytopes as $Q$ and $R$,
respectively, the polynomial
\begin{align*}
f^\prime(x_1,x_2) \coloneqq Q^\prime(x_1,x_2)(x_1 - \alpha) + R^\prime(x_1),
\end{align*}
generally has a Newton polytope strictly larger than~$\NP(f)$.
To see this let $\beta_1$ be the coefficient of the monomial $x_2^2$ in
$Q^\prime$ and $\beta_2$ the coefficient of the same monomial in
$R^\prime$. Then the coefficient of $x_2^2$ in $f^\prime$ equals
$- \alpha \beta_1 + \beta_2$, while $f$ does not contain the monomial $x_2^2$
at all (see Figure~\ref{fig:poly_div}).
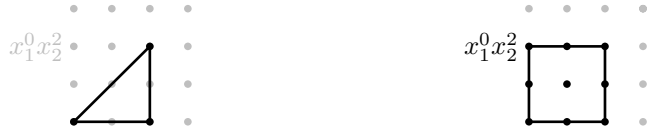
\begin{figure}
	\begin{tikzpicture}[newton polytope]
		\foreach \x in {0,1,2,3}{
			\foreach \y in {0,1,2,3}{
				\fill[lattice point] (\x,\y) circle;
			}
		}
		\draw (0,0) -- (2,0) -- (2,2) -- cycle;
		\foreach \pos in {(0,0), (2,0), (2,2)}{
			\fill \pos circle;
		}
		\node[left, lattice point] at (0,2) {$x_1^0x_2^2$};
	\end{tikzpicture}
	\hfil
	\begin{tikzpicture}[newton polytope]
		\foreach \x in {0,1,2,3}{
			\fill[lattice point]
				(\x,3) circle[]
				(3,\x) circle[];
		}
		\draw (0,0) -- (2,0) -- (2,2) -- (0,2) -- cycle;
		\foreach \x in {0,1,2}{
			\foreach \y in {0,1,2}{
				\fill (\x,\y) circle;
			}
		}
		\node[left] at (0,2) {$x_1^0x_2^2$};
	\end{tikzpicture}
	\caption{Newton polytopes $\NP(f) \not\ni (0,2)$ and $\NP(f') \ni (0,2)$ of Example~\ref{ex:poly_div} for generic $Q'$ and $R'$.}
	\label{fig:poly_div}
\end{figure}
\end{ex}

\begin{defn}
	\label{d:saturated}
	A lattice polytope $P \subset \R_{\geq 0}^d$ is
	\emph{$i$-saturated} if
        $\Rm_i(P) = P \cap \{\langle e_i, \cdot \rangle = 0\}$.
	The polytope $P$ is \emph{$(1,\dots,k)$-saturated} if
	\begin{enumerate}
		\item $P$ is $1$-saturated, and
		\item $\Rm_{(1,\dots,i)}(P) \coloneqq \Rm_i(\Rm_{i-1}(\dots(\Rm_1(P))))$ is $(i+1)$-saturated for all $i \in [k-1]$.
	\end{enumerate}
	Finally, $P$ is \emph{saturated} if it is
        $(1,\dots,d)$-saturated.  All definitions here extend to
        polynomials considering their Newton polytopes.
\end{defn}

See Figures~\ref{fig:saturated_polygon:1} and \ref{fig:saturated_polygon:2} for an example.

\begin{figure}
	\begin{tikzpicture}[newton polytope]
		\foreach \x in {0,1,2,3}{
			\foreach \y in {0,1,2,3}{
				\fill[lattice point] (\x,\y) circle;
			}
		}
		\draw (0,0) -- (2,0) -- (0,2) -- cycle;
		\foreach \pos in {(0,0), (1,0), (2,0), (1,1), (0,2), (0,1)}{
			\fill \pos circle;
		}
		\node[below] at (1,0) {$P\mathstrut$};
	\end{tikzpicture}
	\hfil
	\begin{tikzpicture}[newton polytope]
		\foreach \x in {0,1,2,3}{
			\foreach \y in {0,1,2,3}{
				\fill[lattice point] (\x,\y) circle;
			}
		}
		\draw (0,0) -- (0,2);
		\foreach \y in {0,1,2}{
			\fill (0,\y) circle;
		}
		\node[below] at (0,0) {$\Rm_1(P)$};
	\end{tikzpicture}
	\hfil
	\begin{tikzpicture}[newton polytope]
		\foreach \x in {0,1,2,3}{
			\foreach \y in {0,1,2,3}{
				\fill[lattice point] (\x,\y) circle;
			}
		}
		\draw (0,0) -- (2,0) -- (1,1) -- (0,1) -- cycle;
		\foreach \x in {0,1,2}{
			\fill (\x,0) circle;
		}
		\foreach \x in {0,1}{
			\fill (\x,1) circle;
		}
		\node[below] at (1,0) {$\Qm_1(P) + [0,e_1]$};
	\end{tikzpicture}
	\caption{%
		Example of a 1-saturated polygon $P$.
	}
	\label{fig:saturated_polygon:1}
\end{figure}
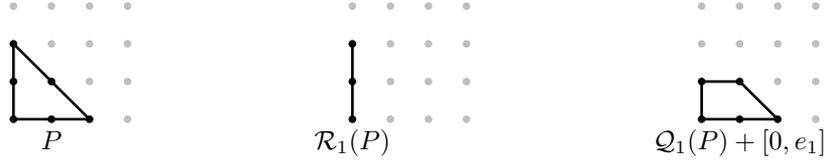

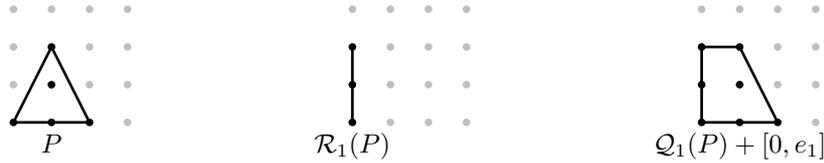
\begin{figure}
	\begin{tikzpicture}[newton polytope]
		\foreach \x in {0,1,2,3}{
			\foreach \y in {0,1,2,3}{
				\fill[lattice point] (\x,\y) circle;
			}
		}
		\draw (0,0) -- (2,0) -- (1,2) -- cycle;
		\foreach \pos in {(0,0), (1,0), (2,0), (1,1), (1,2)}{
			\fill \pos circle;
		}
		\node[below] at (1,0) {$P\mathstrut$};
	\end{tikzpicture}
	\hfil
	\begin{tikzpicture}[newton polytope]
		\foreach \x in {0,1,2,3}{
			\foreach \y in {0,1,2,3}{
				\fill[lattice point] (\x,\y) circle;
			}
		}
		\draw (0,0) -- (0,2);
		\foreach \y in {0,1,2}{
			\fill (0,\y) circle;
		}
		\node[below] at (0,0) {$\Rm_1(P)$};
	\end{tikzpicture}
	\hfil
	\begin{tikzpicture}[newton polytope]
		\foreach \x in {0,1,2,3}{
			\foreach \y in {0,1,2,3}{
				\fill[lattice point] (\x,\y) circle;
			}
		}
		\draw (0,0) -- (2,0) -- (1,2) -- (0,2) -- cycle;
		\foreach \x in {0,1,2}{
			\fill (\x,0) circle;
		}
		\foreach \x in {0,1}{
			\fill (\x,1) circle;
			\fill (\x,2) circle;
		}
		\node[below] at (1,0) {$\Qm_1(P) + [0,e_1]$};
	\end{tikzpicture}
	\caption{%
		Example of a polygon $P$ that is not 1-saturated, but is 2-saturated.
	}
	\label{fig:saturated_polygon:2}
\end{figure}

\begin{rem}
	\label{rem:saturation}
	The polytope $P$ is $i$-saturated if and only if either
        $\langle x,e_i \rangle = 0$, for all $x \in P \cap \Z^d$, or
        if one has $P = \conv((\Qm_i(P) + [0,e_i]) \cup \Rm_i(P))$.
        Here the inclusion
        $P \subseteq \conv((\Qm_i(P) + [0,e_i]) \cup \Rm_i(P))$ always
        holds.
\end{rem}



The polytopes $P_j$ in the following lemma are those arising from
lifting by division and consist of $\NP(f_j - f_j^u)$, the remainder
and several different polytopes arising from substitutions.  They are
contained in $\R^{d+k}$ where the first $d$ corresponds to the
original variables and $d+1,\dotsc,d+k$ to the lifting variables
$y_{1},\dotsc, y_{k}$ in Definition~\ref{d:lift}.  This lemma is the
polyhedral version of the long division by the linear terms
$x_{i}-\alpha_{i}$.
\begin{lem}
  \label{lem:technical_lemma_from_hell}
  Let $k \in [d-1]$. Set $u \coloneqq -e_d$ and let $f_1,\dots,f_d \in \C[x_1,\dots,x_d]$ be polynomials such that each $f_i$ contains a term not involving any of the variables $x_{k+1}, \ldots, x_d$. Assume that all $f_1^u,\dots,f_d^u$ are
  $(1, \dots, k)$-saturated. For $j \in [d]$ we let
  \begin{align*}
    P_j \coloneqq \conv \left(\NP(f_j - f_j^u) \cup \Rm_{(1,\dots,k)}(\NP(f_j^u)) \cup \bigcup_{i=0}^{k-1} (\Qm_{i+1}(\Rm_{(1,\dots,i)}(\NP(f_j^u)))+e_{d+1+i}) \right) \subset \R^{d+k},
  \end{align*}
  where $\Rm_{(1,\dots,0)}$ is the identity and all polytopes in
  $\R^{d}$ are embedded into the first $d$ coordiantes of $\R^{d+k}$.
  Then
  \begin{align*}
    \MV(P_1, \dots, P_d, [0,e_1], \dots, [0,e_k]) = \MV(N_1, \dots, N_d)
  \end{align*}
  with
  \begin{align*}
    N_j \coloneqq \conv(\pi(\NP(f_j)), 0, \pi(e_{d+l}): l \in [k], x_l \text{ occurs in } f_j^u) \subset \R^{d},
  \end{align*}
  where $\pi\colon \R^{d+k} \to \R^d$ is the projection forgetting
  the first $k$ coordinates.
\end{lem}

\begin{proof}
  By \cite[Theorem~5.3.1]{Schn14} we have
  \[
    \MV(P_1, \dots, P_d, [0,e_1], \dots, [0,e_k]) = \MV(\pi(P_1),\dots,\pi(P_d)).
  \]
  By construction of~$P_{j}$, one has
  \begin{align*}
    \pi(P_j) \cap \{\langle \pi(e_{k+1}), \cdot \rangle = \cdots = \langle \pi(e_d), \cdot \rangle = 0\} =
    \conv(0,\pi(e_{d+l}) \colon x_l \text{ occurs in } f_j^u),
  \end{align*}
  finishing the proof upon noting that
  $\Rm_{(1,\dots,k)}(\NP(f_j^u)) = \pi(\NP(f_j^u))$ (which also contains $0$) because
  $f_{1}^{u},\dotsc, f_{d}^{u}$ are $(1,\dotsc,k)$-saturated.
\end{proof}

\begin{rem}
  Assume the $f_{i}$ can be reordered so that $x_l$ occurs in $f_l^u$
  for all $l = 1, \ldots, k$.  Then if all $\NP(f_1),\dots,\NP(f_d)$
  are full-dimensional, it follows that $\MV(N_1, \dots, N_d) > 0$.
  This follows from the definition of the $N_{j}$,
  Definition~\ref{d:essentiality}(ii) and the fact that all
  $\pi(\NP(f_i))$ span $\R^{d-k} \times \{0\}^k$.  
\end{rem}

Our main theorem follows now.  It makes several special choices, some
of which can be assumed without loss of generality.  A discussion
follows the proof in Remark~\ref{r:wlog}.  For example, we only
consider facial subsystems in direction of the last coordinate.
\begin{thm}\label{thm:main}
  Let $k \in [d-1]$ and $u \coloneqq -e_d$. Assume
  $f_1,\dots,f_d \in \C[x_1,\dots,x_d]$ satisfy
  \begin{enumerate}
  \item $f_{i}$ contains a term not involving any of the variables $x_{k+1}, \ldots, x_{d}$ for all $i \in [d]$,
  \item $f_1^u,\dots,f_d^u$ are saturated,
  \item $\MV(N_1, \dots, N_d) > 0$, where the $N_j$ are as in
    Lemma~\ref{lem:technical_lemma_from_hell} and
  \item the system $f_1^u = \dots = f_d^u = 0$ has a
    solution in $(\C^*)^k \times \{0\}^{d-1-k}$ for some
    $k \in [d-1]$.
  \end{enumerate}
  Then there exists a lifting of $(f_1,\dots,f_d)$ to
  $\tilde{f}_1,\dots,\tilde{f}_d, h_1,\dotsc, h_k \in
  \C[x_1,\dots,x_d, y_1, \dots, y_k]$, explicitly given in the proof,
  satisfying
\begin{align*}
\MV(\NP(\tilde{f}_1),\dots,\NP(\tilde{f}_d),\NP(h_1),\dots,\NP(h_k)) < \MV(\NP(f_1),\dots,\NP(f_d)).
\end{align*}
\end{thm}

\begin{proof}
The polynomials $f_1^u,\dots,f_d^u$ are a system in the variables $x_1,\dots,x_{d-1}$
and have a common root $\alpha \in (\C^*)^k \times \{0\}^{d-1-k}$. We define a map
$\phi_{(f_1,\dots,f_d)}$ that sends any tuple
$(g_1,\dots,g_d) \in \C[\NP(f_1),\dots,\NP(f_d)]$ to the tuple
\begin{align*}
\tilde{g}_1(x_1,\dots,x_d,y_1,\dots,y_k) &= y_1 q^1_1 + \dots + y_k q^k_1 + r_1 + (g_1 - g_1^u), \\
&\shortvdotswithin{=}
\tilde{g}_d(x_1,\dots,x_d,y_1,\dots,y_k) &= y_1 q^1_d + \dots + y_k q^k_d + r_d + (g_d - g_d^u), \\
h_1(x_1,\dots,x_d,y_1,\dots,y_k) &= y_1 - (x_1 - \alpha_1), \\
&\shortvdotswithin{=}
h_{k}(x_1,\dots,x_d,y_1,\dots,y_k) &= y_k - (x_k - \alpha_k).
\end{align*}
where $q^1_i,\dots,q^k_i$ are the quotients from polynomial division of $g_i^u$ by $(x_1 - \alpha_1),\dots,(x_k - \alpha_k)$ in this order (and with
respect to an arbitrary monomial order). The polynomial $r_i$ is the remainder and therefore a polynomial only in the variables $x_{k+1}, \dots, x_{d-1}$.
It is straightforward to see that $\phi_{(f_1,\dots,f_d)}$ maps a system to a lifting
in the sense of Definition~\ref{d:lift}.
For $i \in [d]$ define $P_i$ as in Lemma~\ref{lem:technical_lemma_from_hell}
and for $j \in [k]$ let
\begin{align*}
	\Delta_j \coloneqq \conv(0, e_j, e_{d+j}) = \NP(h_j).
\end{align*}
By Lemma~\ref{lem:sparse_division}, the image of
$\phi_{(f_1,\dots,f_d)}$ lies in the vector space
$\C[P_1,\dots,P_d,\Delta_1,\dots,\Delta_k]$.  Let
$U \subset \C[P_1,\dots,P_d,\Delta_1,\dots,\Delta_k]$ be the
Zariski-open subset for which the coefficients of all $y_i$ in $h_i$
are non-zero.  Consider the map which solves the $h_{i}$, giving
$y_{i} = x_{i} - \alpha_{i}$, and substitutes this in the first $d$
equations. Since the facial subsystem $f_1^u,\dots,f_d^u$ is
saturated, this is a well-defined map
$\tilde{\phi}_{(f_1,\dots,f_d)} \colon U \rightarrow
\C[\NP(f_1),\dots,\NP(f_d)]$.  It is also solution-preserving.
Proposition~\ref{prop:trivial_lift} then yields the equality
$\MV(P_1,\dots,P_d,\Delta_1,\dots,\Delta_k) = \MV(\NP(f_1),\dots,\NP(f_d))$.

It is left to show that the Newton polytopes of the system
\begin{align*}
  \phi_{(f_1,\dots,f_d)}(f_1,\dots,f_d) =: (\tilde{f}_1,\dots,\tilde{f}_d,h_1,\dots,h_k)
\end{align*}
have strictly lower mixed volume than
$(P_1,\dots,P_d,\Delta_1,\dots,\Delta_k)$.  Set
$v \coloneqq - (e_{k+1} + \dots + e_d + \dots + e_{d+k}) \in \R^{d+k}$.  Since
$\alpha$ is a solution of $(f_1^u,\dots,f_d^u)$, we have
$r_{1}(\alpha) = \dotsb = r_{d}(\alpha) = 0$.  Hence, the polynomials
$\tilde{f}_1,\dots,\tilde{f}_d$ do not have a constant term and their
Newton polytopes do not contain the origin.  Therefore, in the
notation of Theorem~\ref{thm:bihan-soprunov}, one has
$T_v = \{d+1,\dots,d+k\}$. Moreover, $\Delta_j^v = [0,e_j]$ for
$j \in [k]$, and so the tuple
$(P_1,\dots,P_d,\Delta_1^v,\dots,\Delta_k^v)$ is essential by
Lemma~\ref{lem:technical_lemma_from_hell}.  This implies
$\MV(\NP(\tilde{f}_1),\dots,\NP(\tilde{f}_d),\NP(h_1),\dots,\NP(h_k))
< \MV(\NP(f_1),\dots,\NP(f_d))$ by Theorem~\ref{thm:bihan-soprunov}.
\end{proof}

\begin{rem}
  The assumption of saturatedness of the lifted facial subsystem is
  necessary, as otherwise the image of
  $\tilde{\phi}_{(f_1,\dots,f_d)}$ generically does not lie in
  $\C[\NP(f_1),\dots,\NP(f_d)]$.
\end{rem}
\begin{rem}
  Analyzing the constraints on the coefficients imposed in the
  polynomial division, it can be seen that the lifting in
  Theorem~\ref{thm:main} decreases the mixed volume by exactly one.
  However, this theorem should not be the final word.
  Theorem~\ref{thm:bivariate_lifting} contains a quantitative version
  in the bivariate case and lifting with more general polynomials
  could be more efficient.
\end{rem}

\begin{rem}\label{r:wlog}
  Several assumptions of Theorem~\ref{thm:main} seem restrictive but
  do not sacrifice generality.  Transforming a system using monomial
  changes of variables allows to apply Theorem~\ref{thm:main} for a
  considerably larger class of systems.  Whenever any facial subsystem
  has a solution, one can transform the system so that this subsystem
  is obtained in direction $u = -e_d$ and contains a constant term. Moreover, this transformation can be chosen so that the entire system consists of polynomials with non-negative exponents.
  Furthermore, saturatedness of the Newton polytopes can sometimes be attained by applying further
  transformations.
  It is an interesting open problem to give an explicit
  characterization of tuples of lattice polytopes that can be
  transformed to a saturated tuple using affine unimodular
  transformations.
\end{rem}

\begin{rem}
\label{rem:caveat}
Definition~\ref{d:lift} allows lifted solutions to lie in
$(\C^*)^d \times \C^k$. In particular, certain torus-solutions of the
original system may correspond to non-torus solutions of the lifting.
This can cause problems in applying the lifting strategy of
Theorem~\ref{thm:main} to solve a concrete non-generic system.  Assume
that, in the setting of the above theorem, the facial subsystem in
direction $u$ has a solution at
$x_1 = \alpha_1, \dots, x_k = \alpha_k$ and there exists a torus
solution $s$ of the full system with $s_i = \alpha_i$ for some
$i \in [k]$. Then $s$ corresponds to a non-torus solution in the
lifting because one of the $y_{i}$ vanishes.  However, these non-torus
solutions are not polyhedral in the sense that a small perturbation of
coefficients turns them into torus solutions.
\end{rem}

\begin{rem}\label{rem:LiftOne}
  One drawback of the lifting in Theorem~\ref{thm:main} is that it adds $k$ polynomials for a mixed
  volume reduction of one.
  In general, it would be interesting to explore the liftings of a given system that only add a single polynomial.
  A first insight is that there is no point in substituting a monomial. Indeed, consider a lifting of the form
  \begin{align*}
    \tilde{g}_1(x_1, \ldots, x_d, y) &= y q_1 + r_1 \\
                                     &\shortvdotswithin{=}
                                       \tilde{g}_d(x_1, \ldots, x_d, y) &= y q_d + r_d \\
    h(x_1, \ldots, x_d, y) &= y - x_1^{a_1}\cdots x_d^{a_d}
  \end{align*}
  where $q_i$ and $r_i$ are the quotients and remainders from
  polynomial division of $g_i$ by $h$, i.e.\ replacing
  $x_1^{a_1}\cdots x_d^{a_d}$ by $y$ where possible.  By the
  projection formula for mixed volumes (see
  e.g.~\cite[Theorem~5.3.1]{Schn14}) such a lifting does
  not alter the mixed volume.  Theorem~\ref{thm:main} substitutes binomials,
  which struck us as the next logical step.  Buying into some
  complications on the polyhedral side, more complicated liftings
  seem feasible.
  For example, it would be interesting to investigate liftings in which the new polynomial is
  of the form $p_1y + p_2$ with $p_1,p_2 \in \C[x_1,\dots,x_d]$ where $p_1$ is not constant.
\end{rem}

\section{Lifting using dependency in the facial subsystem}

In this section we present a different type of lifting that does not
assume the facial subsystem to be saturated.  It does not use
polynomial division as in Theorem~\ref{thm:main}, but
exploits dependencies in the facial subsystem.  The lifting replaces two linearly dependent facial
polynomials by a new variable.  For simplicity, we formulate it for
$f^{u}_{1} = \lambda f^{u}_{2}$ but more complicated dependencies of
the $f_{i}^{u}$ can also be exploited under more restrictive conditions. See Remark~\ref{r:lindep}.

\begin{thm}\label{thm:linearly_dependent}
  Let $f_1,\dots,f_d \in \C[x_1^{\pm 1},\dots,x_d^{\pm 1}]$ and
  $u \in \R^d \setminus \{0\}$ such that
  furthermore:
\begin{enumerate}
\item $f_1^u = \lambda f_2^u$,
\item \label{en:ess_ass} $\MV(\NP(f_2^u),\NP(f_3^u),\dots,\NP(f_d^u)) > 0$,
\item \label{en:notfacial} $f_1 \neq f_1^u$.
\end{enumerate}
Then there exists a lift of the system with polynomials $\tilde{f}_1,\tilde{f}_2,f_3,\dots,f_d,h \in \C[x_1^{\pm 1},\dots,x_{d+1}^{\pm 1}]$, explicitly given in the proof, satisfying
\begin{equation*}
	\MV(\NP(\tilde{f}_1),\NP(\tilde{f}_2),\NP(f_3),\dots,\NP(f_d),\NP(h)) < \MV(\NP(f_1),\dots,\NP(f_d)).
\end{equation*}
\end{thm}

\begin{proof}
  Without loss of generality, potentially after a monomial change of
  variables, we assume $u = -e_d$ and
  $\max_{p \in \NP(f_i)}\left<u,p\right>=0$ for all $i \in [d]$.  Then
  the polynomials $f_1^u,\dots,f_d^u$ are a system in the variables
  $x_1,\dots,x_{d-1}$.  Let $\phi_{(f_1,\dots,f_d)}$ send any
  system $(g_1,\dots,g_d) \in \C[\NP(f_1),\dots,\NP(f_d)]$ to the
  system $(\tilde{g}_1,\tilde{g}_2,g_3,\dots,g_d,h)$, where:
\begin{align*}
\tilde{g}_1(x_1,\dots,x_d,y) &= y + g_1^u -f_1^u + (g_1 - g_1^u), \\
\tilde{g}_2(x_1,\dots,x_d,y) &= \lambda y + g_2^u - f_2^u + (g_2 - g_2^u), \\
h(x_1,\dots,x_d,y) &= y - f_1^u.
\end{align*}
One verifies $\phi_{(f_1,\dots,f_d)} (g_{1},\dotsc, g_{d})$ is a lift
in the sense of Definition~\ref{d:lift}.
For $i = 1,2$, let
\begin{align*}
P_i &\coloneqq \conv \left( \{e_{d+1}\} \cup \left(\NP(f_i) \times \{0\} \right) \right),
\end{align*}
and furthermore
\begin{align*}
\Delta &= \conv \left( \{e_{d+1}\} \cup \left( \NP(f_1^u)
\times \{0\} \right) \right).
\end{align*}
Then the image of $\phi_{(f_1,\dots,f_d)}$ lies in the vector space
$\C[P_1,P_2,\NP(f_3),\dots,\NP(f_d),\Delta]$.  Let
$U \subset \C[P_1,P_2,\NP(f_3),\dots,\NP(f_d),\Delta]$ be the
Zariski-open subset of systems for which the coefficient of $y$ in $h$
does not vanish.  Then the map
$\tilde{\phi}_{(f_1,\dots,f_d)} \colon U \to
\C[\NP(f_1),\dots,\NP(f_d)]$, solving $h = 0$ and substituting $y$
into the first two polynomials is solution preserving.
Proposition~\ref{prop:trivial_lift} therefore implies
$\MV(P_1,P_2,\NP(f_3),\dots,\NP(f_d),\Delta)=
\MV(\NP(f_1),\dots,\NP(f_d))$.

It is left to show that the Newton polytopes of the system
\begin{align*}
\phi_{(f_1,\dots,f_d)}(f_1,\dots,f_d)=(\tilde{f}_1,\tilde{f}_2,f_3,\dots,f_d,h),
\end{align*}
have strictly lower mixed volume than
$(P_1,P_2,\NP(f_3), \dots,\NP(f_d),\Delta)$. Set
$v = - e_d - e_{d+1}$.  By construction, $\NP(\tilde{f}_1)$ and
$\NP(\tilde{f}_2)$ only contain monomials in which $x_d$ occurs.  All
other monomials are subtracted in the construction of the lifting.
Then, in the notation of Theorem~\ref{thm:bihan-soprunov}, one has
$T_v = \{ 3,\dots,d+1 \}$.  Moreover one has
\begin{align*}
(P_1,P_2,\NP(f_3)^v,\dots,\NP(f_d)^v,\Delta^v)
= (P_1,P_2,\NP(f_3^u),\dots,\NP(f_d^u),\NP(f_2^u)),
\end{align*}
which is essential.  To see this we use assumption \eqref{en:ess_ass}
to produce $d-1$ linearly independent segments in
$\NP(f_3^u),\dots,\NP(f_d^u),\NP(f_2^u)$.  Additionally there are one
segment in $P_1$ between a point in $\NP(f_1^u)$ and a point in
$\NP(f_1)$ (by \eqref{en:notfacial}), and a segment in $P_2$ between a
point in $\NP(f_2^u)$ and the point~$e_{d+1}$.  The whole collection
of segments is linearly independent.  Theorem~\ref{thm:bihan-soprunov} implies
\begin{equation*}
\MV(\NP(\tilde{f}_1),\NP(\tilde{f}_2),\NP(f_3),\dots,\NP(f_d),\Delta)
< \MV(P_1,P_2,\NP(f_3),\dots,\NP(f_d),\Delta). \qedhere
\end{equation*}
\end{proof}

\begin{rem}\label{r:lindep}
  Theorem~\ref{thm:linearly_dependent} can be generalized to other
  linear dependencies among the $f_{i}^{u}$.  If for some $i \in [d]$
  one has $f_{i}^{u} = \sum_{j \neq i}\lambda_{j}f_{j}^{u}$ and
  $\NP(f_{i}^{u}) \supset \NP(f_{j}^{u})$ for all $j\neq i$, then the
  resubstitution using $\tilde\phi$ has the correct codomain.  Making
  the formulation more technical, also more general linear dependencies
  could be used for lifting.
\end{rem}

\begin{ex}
	Consider the following polynomial system in three variables:
	\begin{align*}
	f_1(x_1, x_2, x_3) &= \underbracket{1 + x_1^2 x_2^2 + x_1^2 x_2^4} + x_3^2 + x_1 x_3 + x_2 x_3, \\
	f_2(x_1, x_2, x_3) &= \underbracket{1 + x_1^2 x_2^2 + x_1^2 x_2^4} + 2x_3^2 + x_1 x_3 + x_2 x_3, \\
	f_3(x_1, x_2, x_3) &= \underbracket{2 + x_1 x_2 + x_1^2 x_2^2 + x_1^2 x_2^4} + x_3^2 + x_1 x_3 + x_2 x_3.
	\end{align*}
	The mixed volume of the Newton polytopes of the $f_i$ is
        $16$. Denote by $f_i^u(x_1, x_2)$ the bracketed subsystems
        which are the facial subsystems with respect to
        $u = (0,0,-1)$.  They have the common solutions
        $(\pm \frac{i}{\sqrt{2}}, \pm i \sqrt{2})$.  It also holds
        that $f_1^u = f_2^u$, so the first condition of
        Theorem~\ref{thm:linearly_dependent} is satisfied. The second
        condition is satisfied as well because the three Newton
        polytopes of the $f_i^u$ agree, so the mixed volume is just
        the lattice volume of some full-dimensional polytope, hence
        positive. The substitution
        $y = f_1^u(x_1, x_2) (= f_2^u(x_1,x_2))$ gives the following
        lifted system in $4$ variables:
	\begin{align*}
	\tilde{f}_1(x_1, x_2, x_3, y) &= y + x_3^2 + x_1 x_3 + x_2 x_3, \\
	\tilde{f}_2(x_1, x_2, x_3, y) &= y + 2x_3^2 + x_1 x_3 + x_2 x_3, \\
	\tilde{f}_3(x_1, x_2, x_3, y) &= f_3(x_1, x_2, x_3), \\
	h(x_1, x_2, y) &= y - (1 + x_1^2 x_2^2 + x_1^2 x_2^4).
	\end{align*}
	This system has mixed volume $12 < 16$. Still, this is not
        best possible as the number of torus solutions is smaller than
        $12$.
\end{ex}

\section{Liftings for bivariate Systems}

In this section we generalize the lifting construction from
Theorem~\ref{thm:main} for bivariate systems, so that it yields a
mixed volume reduction of more than one.  In the bivariate case the
facial subsystems are univariate and thus saturatedness is not an
issue.  For the proof of the following theorem we need another formula
for the mixed volume which can be found, e.g., in \cite{Schn14} and
\cite[Remark~2.3]{BS19}.  We only state it in the case of three
lattice polytopes $P_1, P_2, P_3 \subseteq \R^3$. For any
$v \in \R^3 \setminus \{0\}$ we denote by $\Ht_{P_1}(v)$ the maximum
of $\langle v, p \rangle$ for all $p \in P_1$. Then
\begin{equation}\label{eq:MV_support_functions}
\MV(P_1, P_2, P_3) = \sum_{\substack{v \in \Z^3 \setminus \{0\}\\ \text{primitive}}} \Ht_{P_1}(v) \MV_{v^\perp}(P_2^v, P_3^v),
\end{equation}
where $\MV_{v^\perp}$ denotes the mixed volume with respect to
the linear subspace $v^\perp$ with its induced lattice, and $P_2^v$
and $P_3^v$ denote the faces of $P_2$ and $P_3$ maximizing $v$, viewed
(after a lattice translation) as subsets of~$v^\perp$.  In contrast to
the proofs of Theorems~\ref{thm:main}
and~\ref{thm:linearly_dependent}, the proof of the following theorem
does not rely on the strict monotonicity criterion of~\cite{BS19}.

%

\begin{thm} \label{thm:bivariate_lifting} Let
  $f_1,f_2 \in \C[x_1,x_2]$ and $u \coloneqq -e_2$.  Assume that
  $f_1^u$ and $f_2^u$ have a non-zero constant term and that
  $m = \deg\gcd(f_1^u, f_2^u) > 0$.  Then there exists a lift
  $\tilde{f}_1,\tilde{f}_2,g \in \C[x_1,x_2,y]$, explicitly given in
  the proof, satisfying
  \begin{align*}
    \MV(\NP(\tilde{f}_1),\NP(\tilde{f}_2),\NP(g)) \leq \MV(\NP(f_1),\NP(f_2)) - m.
  \end{align*}
\end{thm}
\begin{proof}
  We adapt the proof of Theorem~\ref{thm:main}.  Here
  $\phi_{(f_1,f_2)}$ maps $g_{1},g_{2}$ to
  \begin{align*}
    \tilde{g}_1(x_1,x_2,y) &= y q_1 + r_1 + (g_1 - g_1^u), \\
    \tilde{g}_2(x_1,x_2,y) &= y q_2 + r_2 + (g_2 - g_2^u), \\
    h(x_1,x_2,y) &= y - \gcd(f_1^u,f_2^u),
  \end{align*}
  where now $q_j$ and $r_j$ are quotient and remainder of the division
  of $f_j^u$ by $\gcd(f_1^u,f_2^u)$.  Again, $\phi_{(f_1,f_2)}$ lifts
  $(g_1, g_2)$ in the sense of Definition~\ref{d:lift}.  The system
  $\phi_{(f_{1},f_{2})}(g_{1},g_{2})$ lies in $\C[P_1,P_2,\Delta]$,
  where
  \begin{align*}
    P_1 &\coloneqq \conv \Big( \{0\} \cup (e_3 +[0,(\deg(f_1^u)-m)e_1]) \cup \NP(f_1-f_1^u) \Big), \\
    P_2 &\coloneqq \conv \Big( \{0\} \cup (e_3 +[0,(\deg(f_2^u)-m)e_1]) \cup \NP(f_2-f_2^u) \Big), \\
    \Delta &\coloneqq \conv \Big( \{e_3\} \cup [0, m e_1] \Big).
  \end{align*}
  Let $U \subset \C[P_1,P_2,P_3]$ be the Zariski-open subset for which
  the coefficient of $y$ in $h$ is non-zero.  The generic
  resubstitution
  $\tilde{\phi}_{(f_1,f_2)} \colon U \rightarrow
  \C[\NP(f_1),\NP(f_2)]$ which solves $h=0$ for $y$ and plugs the
  result into $\tilde{g}_1$ and $\tilde{g}_2$ is well defined.
  Therefore, Proposition~\ref{prop:trivial_lift} yields the equality
  $\MV(P_1,P_2,\Delta) = \MV(\NP(f_1),\NP(f_2))$.  We now show that
  the mixed volume decreases by~$m$ upon lifting. The remainder $r_i$
  in $\tilde{f}_i$ vanishes for $i=1,2$.  Without loss of generality
  one can assume that $x_2$ occurs in~$f_2$ as it occurs in $f_{1}$ or
  $f_{2}$ or we had a univariate system to start with.  Let $k \geq 1$
  be minimal such that $\NP(\tilde{f}_2)$ contains a lattice point
  with $x_2$-coordinate~$k$. Let $w \coloneqq - e_2 - k e_3 \in \R^3$.
  Using \eqref{eq:MV_support_functions}, we have the following chain
  of inequalities:
  \begin{align*}
    \MV(P_1,P_2,\Delta) - \MV(\NP(\tilde{f}_1), \NP(\tilde{f}_2), \Delta) &\geq \MV(P_1, \NP(\tilde{f}_2), \Delta) - \MV(\NP(\tilde{f}_1), \NP(\tilde{f}_2), \Delta) \\ &= \sum_{\substack{v \in \Z^3 \setminus \{0\}\\ \text{primitive}}} \left( \Ht_{P_1}(v) - \Ht_{\NP(\tilde{f}_1)}(v) \right) \MV_{v^\perp}(\NP(\tilde{f}_2)^v, \Delta^v) \\ &\geq \left(\Ht_{P_1}(w) - \Ht_{\NP(\tilde{f}_1)}(w)\right) \MV_{w^\perp}(\NP(\tilde{f}_2)^w, [0, m e_1]) \\ &\geq m \cdot \MV(\pi(\NP(\tilde{f}_2)^w)) = m.
	\end{align*}
	Here, $\pi\colon w^\perp \rightarrow \R$ is the projection along the $x_1$-axis, i.e., forgetting the $e_1$-coordinate observing that $e_1 \in w^\perp$. Moreover, we use weak monotonicity of the mixed volume in the first step, the fact that $\Ht_P(v) \geq \Ht_Q(v)$ for $P \supseteq Q$ in the third step, and $\MV(\pi(N(\tilde{f}_2)^w)) = 1$ by definition of $w$ in the last step.
\end{proof}

\begin{rem}
  In order to determine a lifting as in
  Theorem~\ref{thm:bivariate_lifting}, it suffices to compute the
  $\gcd$ of two univariate polynomials via the Euclidean algorithm. In
  particular, it is not necessary to compute the common roots of the
  facial subsystem as in Theorem~\ref{thm:main}.
\end{rem}

\bibliographystyle{alpha}
\bibliography{decreasing_mixed_volume}

\begin{thebibliography}{CLO05}

\bibitem[AD10]{albersmeyer2010lifted}
Jan Albersmeyer and Moritz Diehl.
\newblock The lifted {N}ewton method and its application in optimization.
\newblock {\em SIAM Journal on Optimization}, 20(3):1655--1684, 2010.

\bibitem[Ber75]{B75}
David~N. Bernstein.
\newblock The number of roots of a system of equations.
\newblock {\em Funkcional. Anal. i Prilo\v{z}en.}, 9(3):1--4, 1975.

\bibitem[BS19]{BS19}
Fr\'{e}d\'{e}ric Bihan and Ivan Soprunov.
\newblock Criteria for strict monotonicity of the mixed volume of convex
  polytopes.
\newblock {\em Advances in Geometry}, 19(4):527--540, 2019.

\bibitem[BT18]{hcJulia}
Paul Breiding and Sascha Timme.
\newblock {HomotopyContinuation.jl:} {A} package for homotopy continuation in
  {J}ulia.
\newblock In {\em International Congress on Mathematical Software}, pages
  458--465. Springer, 2018.

\bibitem[CLO05]{CLO05}
David~A. Cox, John Little, and Donal O'Shea.
\newblock {\em Using algebraic geometry}, volume 185 of {\em Graduate Texts in
  Mathematics}.
\newblock Springer, New York, second edition, 2005.

\bibitem[HS95]{huberSturmfels}
Birkett Huber and Bernd Sturmfels.
\newblock A polyhedral method for solving sparse polynomial systems.
\newblock {\em Mathematics of computation}, 64(212):1541--1555, 1995.

\bibitem[Lai17]{lairez2017deterministic}
Pierre Lairez.
\newblock A deterministic algorithm to compute approximate roots of polynomial
  systems in polynomial average time.
\newblock {\em Foundations of computational mathematics}, 17(5):1265--1292,
  2017.

\bibitem[Sch14]{Schn14}
Rolf Schneider.
\newblock {\em Convex bodies: the {B}runn-{M}inkowski theory}, volume 151 of
  {\em Encyclopedia of Mathematics and its Applications}.
\newblock Cambridge University Press, Cambridge, expanded edition, 2014.

\bibitem[Ver11]{phcpack}
Jan Verschelde.
\newblock Polynomial homotopy continuation with {PHCpack}.
\newblock {\em ACM Communications in Computer Algebra}, 44(3/4):217--220, 2011.

\end{thebibliography}
\end{document}